\documentclass[12pt]{amsart}
\usepackage[margin=1in]{geometry}
\usepackage{amssymb}
\usepackage{amsmath,amscd}
\usepackage[T1]{fontenc}
\usepackage{url}
\usepackage{enumitem}
\usepackage{pinlabel}
\usepackage{graphicx}

\theoremstyle{plain}
\newtheorem{theorem}{Theorem}[section]
\newtheorem{proposition}[theorem]{Proposition}
\newtheorem{lemma}[theorem]{Lemma}

\theoremstyle{definition}
\newtheorem{definition}[theorem]{Definition}

\theoremstyle{remark}
\newtheorem{remark}[theorem]{Remark}

\newtheorem*{Conc4}{Condition $C^{\prime \prime} (4)$}
\newtheorem*{Cont4}{Condition $T (4)$ }

\def\BZ{\mathbb Z}

\def\A{\mathcal A}

\def\O{\mathcal O}

\def\calP{\mathcal P}

\def\calT{\mathcal T}

\def\la{\langle}
\def\ra{\rangle}

\def\PD{\Delta}
\def\squarep{square presentation}
\def\gridp{grid presentation}

\def\pcomplex{peripheral complex}

\def\calD{\mathcal D}

\def\Z{\mathbb Z}
\def\c4t4{$C^{\prime \prime} (4) - T (4)$ }
\def\calP{\mathcal P}
\def\canc4{$C^{\prime \prime} (4)$}
\def\cant4{$T (4)$}

\begin{document}

\title[Non-peripheral ideal decompositions of alternating 
knots]{Non-peripheral ideal decompositions of alternating knots}

\author{Stavros Garoufalidis}
\address{School of Mathematics \\
         Georgia Institute of Technology \\
         Atlanta, GA 30332-0160, USA \newline
         {\tt \url{http://www.math.gatech.edu/~stavros}}}
\email{stavros@math.gatech.edu}
\author{Iain Moffatt}
\address{Department of Mathematics \\ 
Royal Holloway \\
         University of London \\
         Egham, Surrey \\
         TW20 0EX, United Kingdom \newline
         {\tt \url{http://www.personal.rhul.ac.uk/uxah/001}}}
\email{iain.moffatt@rhul.ac.uk}
\author{Dylan P. Thurston}
\address{Department of Mathematics \\
        Indiana University \\ 
        Bloomington, IN 47405-7106, USA \newline
        {\tt \url{http://pages.iu.edu/~dpthurst}}}
\email{dthurston@indiana.edu}

\thanks{S.G. and D.T. were supported in part by National Science
  Foundation Grants DMS-15-07244 and DMS-14-06419 respectively. 
}
\subjclass[2010]{Primary 57N10. Secondary 20F06, 57M25.}
\keywords{ideal triangulations, knots, hyperbolic
  geometry, ideal tetrahedra, small cancellation theory, 
  Dehn presentation, alternating knots, Volume Conjecture.}

\date{\today}


\begin{abstract}
An ideal triangulation $\calT$ of a hyperbolic 3-manifold $M$ with one
cusp is non-peripheral if no edge of $\calT$ is homotopic to a curve 
in the boundary torus of $M$. For such a triangulation, the gluing and
completeness equations can be solved to recover the hyperbolic structure
of $M$. A planar projection of a knot gives four ideal cell
decompositions of its complement (minus 2 balls), two of which are ideal 
triangulations that use 4 (resp., 5) ideal tetrahedra per crossing. 
Our main result is that these ideal triangulations are 
non-peripheral for all planar, reduced, alternating projections of 
hyperbolic knots. 
Our proof uses the small cancellation properties of the 
Dehn presentation of  alternating knot groups, and an explicit solution
to their word and conjugacy problems. 
In particular, we describe a planar complex that encodes all geodesic words that represent elements of the peripheral subgroup of an alternating knot group.  This gives a polynomial time algorithm for checking if an element in an alternating knot group is peripheral.
Our motivation for this work comes from the Volume 
Conjecture for  knots.
\end{abstract}

\maketitle
\tableofcontents


\section{Introduction}
\label{sec.intro}

\subsection{Motivation: the Volume Conjecture}
\label{sub.motivation}

The motivation of our paper comes from the Kashaev's {\em Volume Conjecture} 
for knots in 3-space, which states that for a hyperbolic knot $K$ in
$S^3$ we have:
$$
\lim_{n \to \infty} \frac{1}{n} \log |\la K\ra_N | =
\frac{\mathrm{Vol}(K)}{2\pi}
$$
where $\la K \ra_N$ is the {\em Kashaev invariants} of $K$; see
\cite{K, MM}. This gives a precise connection between quantum topology
and hyperbolic geometry. The Volume Conjecture has been verified for only a handful 
hyperbolic knots: initially for the simplest hyperbolic $4_1$ knot
and now, due to the work of Ohtsuki~\cite{Oh1}, and Ohtsuki and Yokota~\cite{Oh2}, for all hyperbolic
knots with at most $6$ crossings.

The Volume Conjecture requires a common input for computing both
the Kashaev invariant and the hyperbolic volume. Such an input turns out
to be a planar projection of a knot $K$ which allows one to
express the Kashaev invariant as a multi-dimensional state sum whose
summand is a ratio of quantum factorials (4 or 5, depending on the
model used).

On the other hand, a planar projection gives four ideal cell
decompositions of its complement (minus 2 balls), two of which are ideal 
triangulations that use 4 (resp., 5) ideal tetrahedra per crossing. 
These ideal triangulations are well-known from the early days of hyperbolic
geometry, and were used by Weeks~\cite{We} (in his computer program
\texttt{SnapPy}~\cite{snappy}), by the third author~\cite{Th1}, 
Yokota~\cite{Yo1, Yo3}, Sakuma-Yokota~\cite{SY} and others.

An approach to the Volume Conjecture initiated by the third author in 
\cite{Th1}, and also by Yokota, Kashaev, Hikami, the first author 
and others (see \cite{Ga1,KY,Hi,Yo1}), is to convert multi-dimensional 
state-sum formulas for the Kashaev invariant to multi-dimensional 
state-integral formulas over suitable cycles, and then to apply a 
steepest descent method to study the asymptotic behaviour of 
the Kashaev invariant. The summand (and hence, the integrand) depends on 
the planar projection and the steepest descend method is applied 
to a leading term of the integrand, the so-called potential function.
The critical points of the potential function have a 
geometric meaning, namely they are solutions to the \emph{gluing equations}.
The latter are a special system of polynomial equations (studied by
W.~Thurston and Neumann-Zagier in \cite{Th,NZ}) that are associated to the 
ideal triangulations of the knot complement discussed above. A suitable
solution to the gluing equations recovers the hyperbolic structure,
and the value of the potential function is the volume of the knot.

The problem is that every planar projection leads to ideal triangulations,
hence to gluing equations, and even if we know that the knot is hyperbolic,
it is by no means obvious that those gluing equations have a suitable solution
(or in fact, any solution) that recovers the complete hyperbolic structure. 
It turns out that if a knot is hyperbolic, the lack of a suitable solution 
occurs only when edges of the ideal triangulation are homotopic
to peripheral curves in the boundary tori.

\subsection{Non-peripheral ideal triangulations of alternating knots}
\label{sub.nonper}

{\em Ideal triangulations} of hyperbolic 3-manifolds with cusps were introduced
by W. Thurston in his study of Geometrization of 3-manifolds; see \cite{Th}.
For thorough discussions, see \cite{BP,snappy,NZ,We}.
An ideal triangulation $\calT$ of a hyperbolic 3-manifold $M$ with one
cusp is {\em non-peripheral} if no edge of $\calT$ is homotopic to a curve 
in the boundary torus of $M$. For such a triangulation, the gluing and
completeness equations of \cite{NZ}
can be solved to recover the hyperbolic structure
of $M$. For a proof, see \cite[Lem.2.2]{Ti} and also the discussion in
\cite[Sec.3]{DG}.

A planar projection $\PD$ of a knot gives rise to four ideal cell
decompositions of its complement (namely, $\calT_{2B}(\PD)$,
$\calT_{O}^\circ(\PD)$, $\calT_{4T}^\circ(\PD)$ and $\calT_{5T}^\circ(\PD)$), 
the last two of which are ideal triangulations that use 4 (resp., 5) ideal 
tetrahedra per crossing. We will briefly recall these decompositions here, 
although their precise definition is not needed for the statement and proof 
of Theorem~\ref{thm.arcs} below.

\noindent
$\bullet$ 
$\calT_{2B}(\PD)$ is a decomposition of the knot complement into one ball 
above and one ball below the planar projection. These two balls have a 
cell-decomposition that matches the planar projection of the knot, and 
were originally studied by W. Thurston, and more recently by 
Lackenby~\cite{La}.

\noindent
$\bullet$
$\calT_{O}^\circ(\PD)$ is a decomposition of the knot complement minus two 
balls into ideal octahedra, one at each crossing of $\PD$. This was described 
by Weeks~\cite{We}, and also by the third author~\cite{Th1}, and by 
Yokota~\cite{Yo1,Yo3}.

\noindent
$\bullet$
Each ideal octahedron can be subdivided into 4 ideal tetrahedra, or into
5 ideal tetrahedra. Thus, a subdivision of $\calT_{O}^\circ(\PD)$ gives rise
to two ideal triangulations of the knot complement minus two balls, denoted
by $\calT_{4T}^\circ(\PD)$ and $\calT_{5T}^\circ(\PD)$.

\begin{theorem}
\label{thm.1}
If $\PD$ is a prime, reduced, alternating projection of a non-torus knot $K$,
then the four ideal cell decompositions $\calT_{2B}(\PD)$,
$\calT_{O}^\circ(\PD)$, $\calT_{4T}^\circ(\PD)$ and $\calT_{5T}^\circ(\PD)$ 
are non-peripheral. Consequently, the gluing equations have a solution
that recovers the complete hyperbolic structure.
\end{theorem}

\subsection{Alternating knots and small cancellation theory}
\label{sub.results}

The above theorem follows from proving that all edges of the above ideal
triangulations are homotopically non-peripheral. Luckily, we can describe 
those edges directly in terms of the planar projection of the knot as follows.

\begin{definition}
Let $\Delta\subset \mathbb{R}^2$ be a knot diagram with $n$
crossings. Consider the projection plane $\mathbb{R}^2 $ as the
$xy$-plane of $\mathbb{R}^3$, and consider the knot $K\subset
S^3=\mathbb{R}^3 \cup \{\infty\}$ obtained from $\Delta$ by
``pulling'' the overcrossing arcs above the plane and undercrossing
arcs under the plane in the standard way. Fix a basepoint for
$\pi_1(S^3 \setminus K)$ in the unbounded region near one strand
of~$K$. We distinguish four kinds of loops in $\pi_1(S^3 \setminus K)$.
\begin{enumerate}
\item A \emph{Wirtinger arc}
  follows the double of $\PD$ through $k$ crossings with $1 < k < 2n$
  and then returns to the basepoint through either the upper or lower
  half-space.
\item A \emph{Wirtinger loop}
  starts at the basepoint, travels in either the upper (resp.\ lower)
  half-space to pass through a region~$R$ of~$\Delta$, passes through a
  region adjacent to~$R$, and then returns through the upper (resp.\ lower)
  half-space to the basepoint. We forbid the
  short loop around the strand near the basepoint, which is manifestly
  a meridian.
\item  A \emph{Dehn arc} starts at the basepoint, travels in the
  upper (resp.\ lower) half-space through a region of $\PD$ and then returns
  to the basepoint through the lower (resp.\ upper) half-space without
  passing through the projection plane.
\item A \emph{short arc} follows the double of $\PD$ from the basepoint
  until some crossing, where it jumps to the other strand in the
  crossing and then follows the double back 
  to the basepoint.
\end{enumerate}
There four types of arc are illustrated in Figure~\ref{fig.arcs}.
\end{definition}

\begin{figure}[ht]
\begin{center}
\begin{tabular}{ccc}
\includegraphics[height=0.2\textheight]{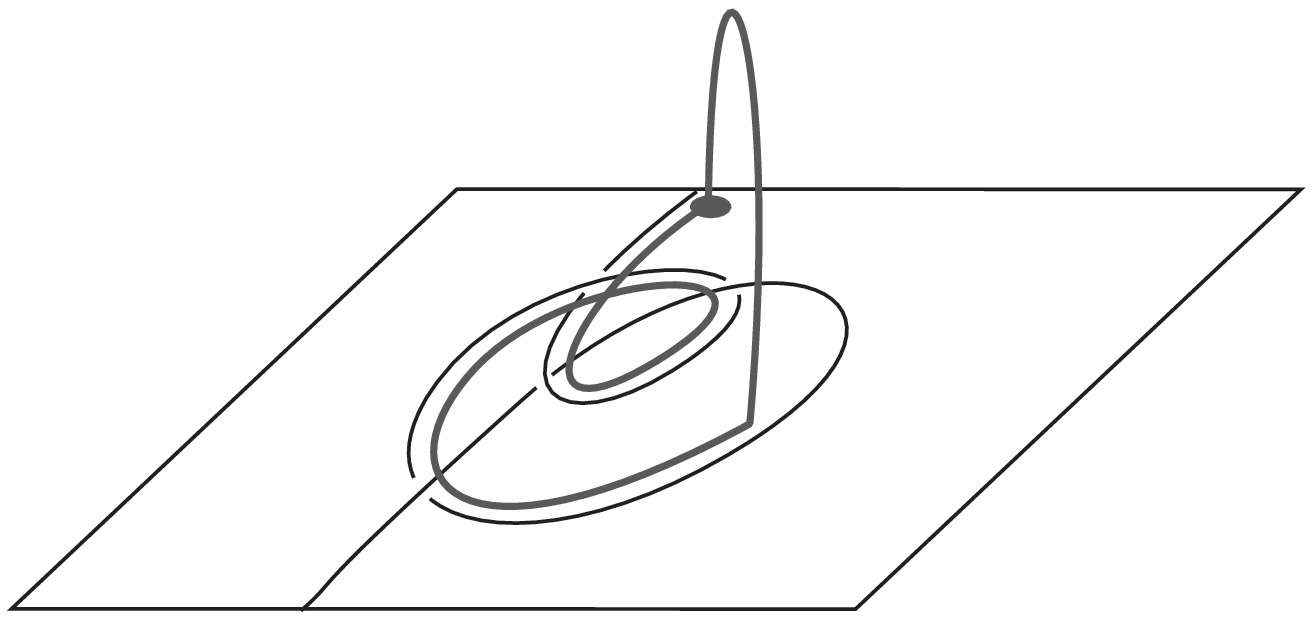}
& \quad\quad &
\includegraphics[height=0.2\textheight]{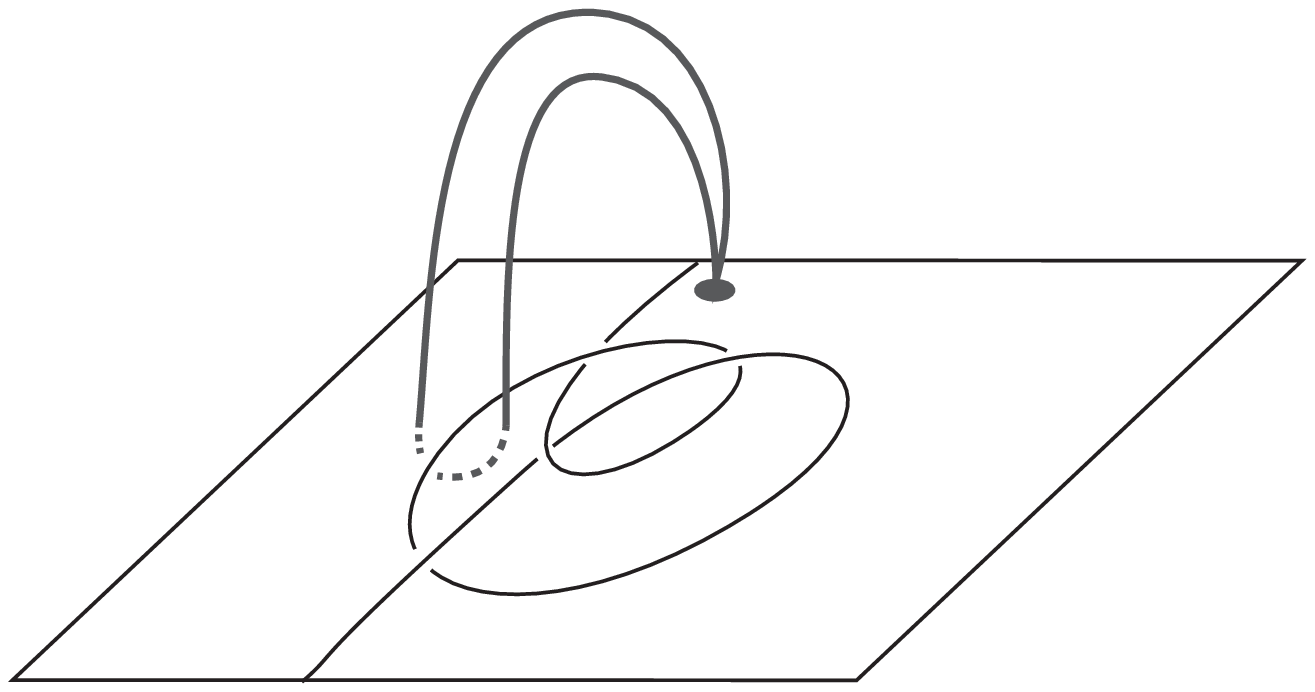}
\\
Wirtinger arc&&Wirtinger loop
\\
\includegraphics[height=0.2\textheight]{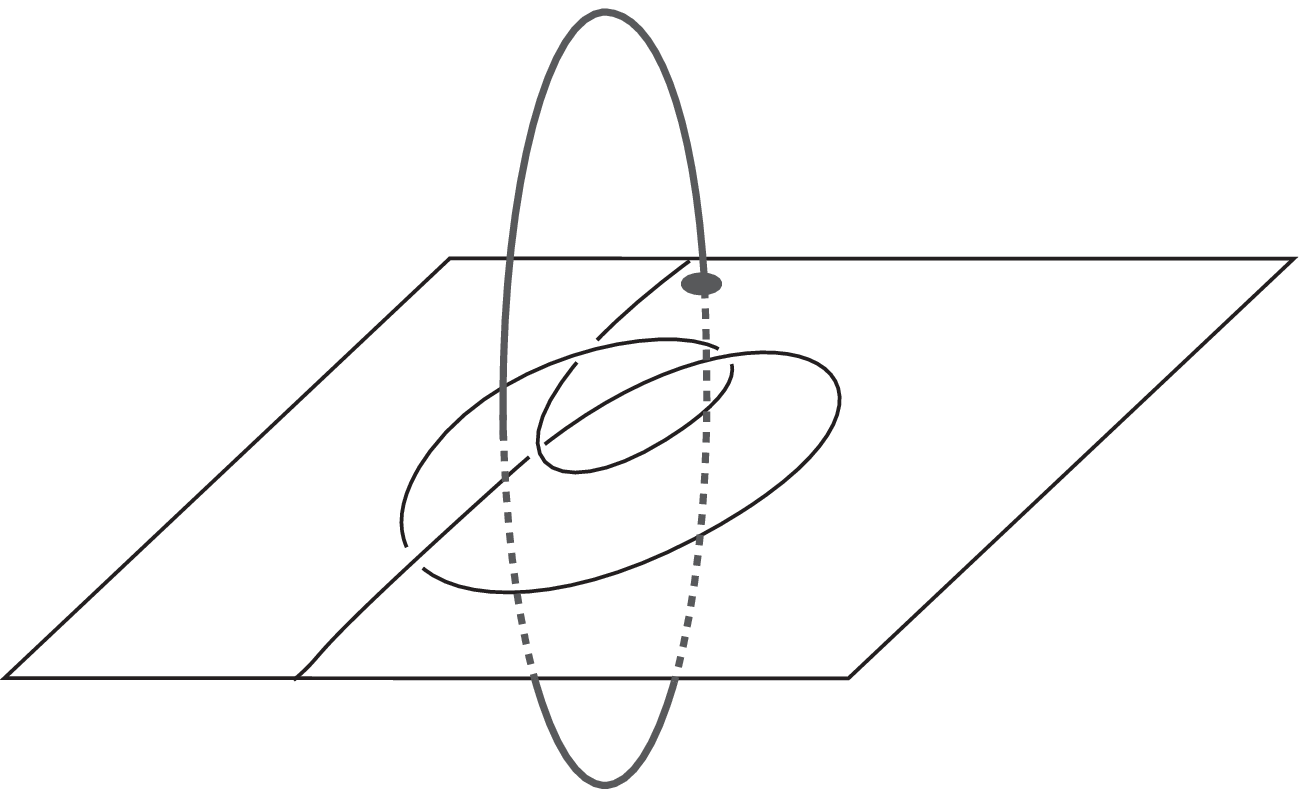}
&&
\includegraphics[height=0.2\textheight]{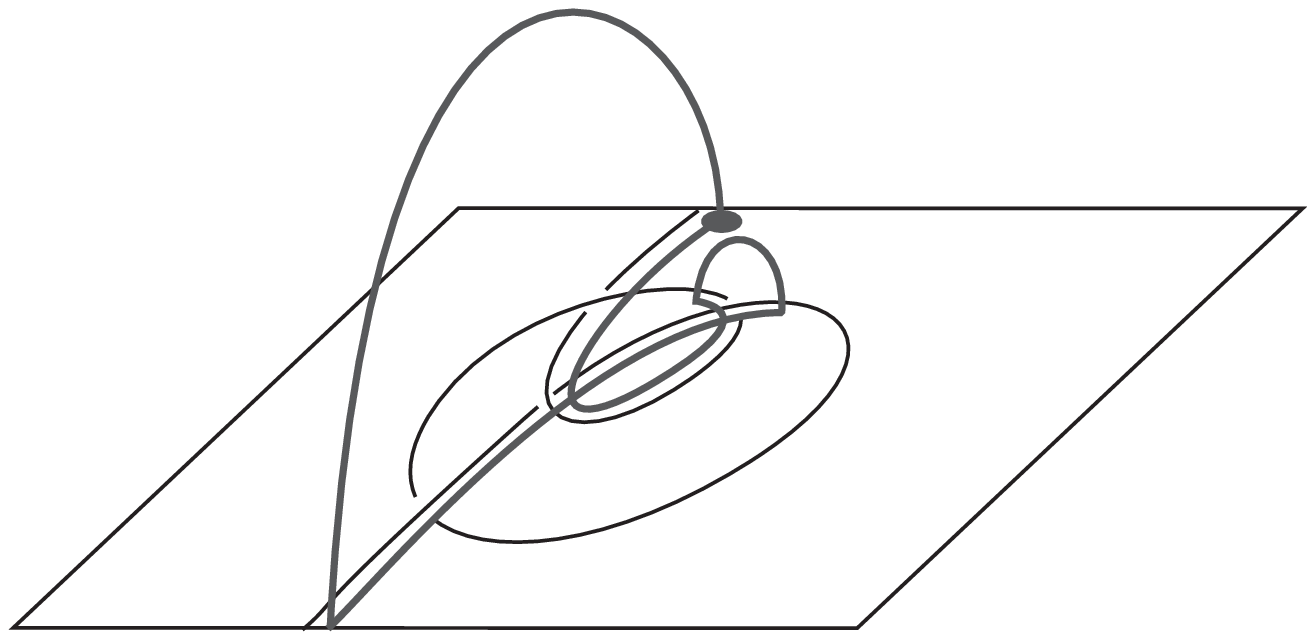}
\\
Dehn arc&& Short arc
\end{tabular}
\end{center}

\caption{Four types of loop in  a knot complement.}
\label{fig.arcs}
\end{figure}

These arcs are denoted by the letters $A$, $B$, $C$ and $D$ in~\cite{SY}.

\begin{theorem}
\label{thm.arcs}
If $\PD$ is a prime, reduced, alternating projection of a non-torus knot $K$,
then all Wirtinger arcs, Wirtinger loops, Dehn arcs and short arcs are
non-peripheral.
\end{theorem}

Theorem~\ref{thm.1} immediately follows from Theorem~\ref{thm.arcs},
since all of the arcs that appear in any of the decompositions in
Theorem~\ref{thm.1} are of one of the four types in
Theorem~\ref{thm.arcs}.

The proof of Theorem~\ref{thm.arcs} uses the small cancellation
property of the
Dehn presentation of hyperbolic alternating knots. Curiously, our proof
uses an explicit solution to the conjugacy problem of the Dehn presentation
of a prime reduced alternating planar projection~$\PD$.  See 
Remark~\ref{rem.conjugacy} below.

\subsection*{Acknowledgements} 
A first draft of this paper was written in 2002 and was completed in 2007,
but unfortunately remained unpublished. During a conference in Waseda 
University in 2016 in honour of the 20th anniversary of the Volume Conjecture, 
an alternative proof of the results of our paper (using cubical complexes) 
was announced by Sakuma-Yokota~\cite{SY}, and with the same motivation as 
ours. We thank Sakuma-Yokota for their encouragement to publish our results, 
and the organisers of the Waseda conference (especially Jun Murakami) 
for their hospitality.


\section{Small cancellation theory}
\label{sec.scancel}

\subsection{The (augmented) Dehn presentation of a knot group}
\label{sub.Dehnp}

We  begin with a discussion of the augmented Dehn presentation of a knot 
diagram. As it turns out, the augmented Dehn presentation (defined below) is a small 
cancellation group and this structure provides a quick and implementable solution to its word problem. Background on small cancellation groups and combinatorial group theory can be found in \cite{LS}.

Throughout this paper we  implicitly {\em symmetrize} all group 
presentations. This means that when we write a set of relators $R$, we 
actually mean the set of all relators which can be obtained from $R$ by 
inversion and cyclic permutation.

Let $\PD$ be a $n$ crossing planar diagram of a link $L$. Of the $n+2$ 
regions 
of the diagram~$\PD$, exactly $n+1$ of these regions are bounded. Assign a 
unique label $1,2, \ldots , n+1$ to each of these bounded region and the 
label $0$ to the unbounded region. We identify each region with its label. 

We obtain a group presentation from the labelled diagram $\PD$ as follows. 
Take one generator   $X_i$ for each region $i=0,1,2, \ldots ,n+1$ of $\PD$. Take 
 one relator   $R_i$ for each of the $n$ crossings of $\PD$  which is read from the diagram thus
%
\begin{center}
 \labellist
\small\hair 2pt
\pinlabel $a$ at  35 71
\pinlabel $b$ at  72 71
\pinlabel $c$ at  72 35
\pinlabel $d$  at 35 35
\endlabellist
\raisebox{-3mm}{\includegraphics[height=25mm]{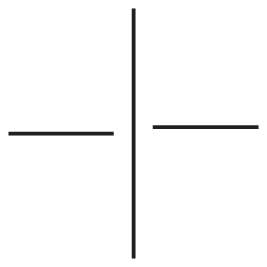}}
\quad
\raisebox{8mm}{$\leadsto \quad X_a X_b^{-1} X_c X_d^{-1}$}
\end{center}
%
If we choose a  base point above the projection plane, and we choose a point 
$p_i$ in the interior of each region $i$. Then the generator~$X_i$ can be 
described geometrically by a loop in the knot complement
which passes from the base point, downwards through the region  $p_i$ then 
back up to the base point through the point $p_0$ which lies in the 
unbounded region. Dehn showed that 
\[
\calD_{\PD} \overset{def}{=}  \left\langle \left. 
 X_0 , X_1, \ldots , X_{n+1} \; \right| \; R_1, R_2, \ldots R_n  , X_0   
 \right\rangle 
\] 
is a presentation for the knot group $\pi_1(S^3\setminus L ) $.  
We call this the {\em Dehn presentation} of $\pi_1(S^3\setminus L ) $ 
read from the diagram $\PD$. In what follows, we  use a minor
modification of the Dehn presentation which has better small cancellation
properties.

The {\em augmented Dehn presentation}, $\A_{\PD}$, of $\PD$ is the group 
presentation
$$   \A_{\PD} \overset{def}{=}  \left\langle \left.  
X_0 , X_1, \ldots , X_{n+1} \; \right| \; R_1, R_2, \ldots R_n     
\right\rangle .  $$

The augmented Dehn presentation arises as a Dehn presentation of a link.
Given a labelled link diagram $\PD$,  construct a new labelled link
diagram $\PD \cup \O$ by adding a zero-crossing component $\O$, 
which bounds $\PD$. This is called the \emph{augmented link diagram}.
The augmented Dehn presentation of $\PD$ is a 
presentation for the for the {\em augmented link group}    
$\pi_1(S^3\setminus(K  \cup \O)) $,  i.e.,
\begin{equation}\label{ahj}
\A_{\PD} \cong \calD_{\PD \cup \O} \cong \pi_1(S^3\setminus K) \ast \Z \,.   
\end{equation}

We will solve the word problem in $\calD_{\PD}$ by solving it in  $\A_{\PD}$. For completeness, let us say a few words about why it is sufficient to solve the word problem in $\A_{\PD}$. This is a consequence of some standard facts about group presentations that can be found in, for example, \cite{LS}. 
Let 
$\calP_G = \left\langle \left.  g_1, \ldots , g_{k} \; \right| \; 
r_1, \ldots r_j     \right\rangle  $
and 
$\calP_H = \left\langle \left.  h_1, \ldots , h_{l} \; \right| \; 
s_1, \ldots s_m     \right\rangle  $
be presentations for groups $G$ and $H$ respectively. Then the 
{\em standard presentation}, which we  denote by  
$\calP_{G}\ast \calP_{ H}$,  for the free product $G\ast H$ is
 
$$
\calP_{G}\ast \calP_{ H} = \left\langle \left.  g_1, \ldots , g_{k},h_1, 
\ldots , h_{l} \; \right| \; r_1, \ldots r_j  , s_1, \ldots s_m   
\right\rangle  .
$$
A standard consequence of the {\em normal form for 
free products} (again see \cite{LS}) is that with $\calP_G$, $\calP_H$ and $\calP_{G}\ast \calP_{ H}$  as above, if $w$ is a word in the generators $g_1, \ldots , g_{k}$ and their inverses, then $w=_G 1$ if and only if $w=_{G\ast H} 1$. Thus, by \eqref{ahj}, the word problem in $\calD_{\PD} \cong \pi_1(S^3\setminus K)$ can be solved by the word problem in $\A_{\PD} \cong \pi_1(S^3\setminus K) \ast \Z$.

An an explicit isomorphism of 
the augmented Dehn presentation with a standard presentation  for the 
free product $\pi_1(S^3\setminus K) \ast \Z$ is given by
\begin{equation}
\label{eq.phi}
\phi : \A_{\PD} \rightarrow   \calD_{\PD} \ast \langle Y  | \;\;  \rangle    
\end{equation}   
where
\[
\phi : X_i \mapsto 
\left\{ \begin{array}{ll}
Y & \text{if } i=0 \\ 
X_iY^{-1} & \text{otherwise }
\end{array}\right. .
\]
 Geometrically, $\phi$  
corresponds to isotoping the component $\O$ of the 
augmented link in $S^3$ away from the subdiagram $\PD$ so that it bounds a disc 
in the projection plane. 

\begin{remark}
\label{rem.phi}
Let $\iota: \calD_{\PD} \to \calD_{\PD} \ast \langle Y  | \;\;  \rangle$
denote the natural inclusion.
Given a projection $l$  of a loop $\ell \in \pi_1 (S^3\setminus K)$ in 
the diagram 
$\PD$,  we can read off a representative $\phi^{-1} (\iota (w))$ as follows:  
follow the loop $l$ from its basepoint in the direction of its orientation. 
When $l$ ``passes downwards'' through a region $i$ of $\PD$ assign a generator 
$X_i$; and whenever $l$ ``passes upwards'' through a region $i$ of $\PD$ 
assign a generator $X_i^{-1}$.   
The  word thus obtained clearly represents the loop $\ell$.
Thus, $w \neq_{\pi_1(S^3\setminus K)} 1$ if and only if
$\phi^{-1} (\iota (w)) \neq_{\A_{\PD}} 1$.
\end{remark}

\subsection{Square and grid presentations}
\label{sub.squaregrid}

The augmented Dehn presentation of a 
prime, reduced, alternating knot diagram has  \emph{small cancellation}
properties, as was first observed by Weinbaum in \cite{Wn}. 

Let $G= \langle X | R \rangle $ be a symmetrized group presentation.  
We call a non-empty word $r$  a {\em piece} with respect to $R$ if there 
exist distinct words $s,t \in  R$ such that $ s = r u$ and $t = rv$. 
\begin{definition}
\label{def.squaregrid}
\rm{(a)} A symmetrized presentation $\la X | R\ra$ is called a {\em 
\squarep  } 
if it satisfies the following two small cancellation conditions:
\begin{Conc4}
All relators have length four and no defining relator is a product of 
fewer than four pieces.
\end{Conc4}
\begin{Cont4}
Let $r_1 , r_2$ and $r_3$ be any three defining relators such that no two 
of the words are inverses to each other, then one of $r_1 r_2$, $r_2 r_3$ 
or $r_3 r_1$ is freely reduced without cancellation.
\end{Cont4}
\rm{(b)} A symmetrized presentation $\la X | R\ra$ is called a {\em 
\gridp } if it is a \squarep\ and in addition $X$ is colored by two colors
(black or white) and every relator alternates in the
two colors and in taking inverse.
\end{definition} 

\begin{remark}
\label{rem.diffnames}
There does not appear to be a standard terminology of the above definition.
In \cite{Wn}, Weinbaum calls \squarep s $C''(4)-T(4)$ presentations.
In \cite{Jg1,Jg2}, Johnsgard uses the term {\em parity} to denote
the black/white coloring of a \gridp . In \cite[Defn.3.1]{Wi1} and 
\cite[Defn.2.2]{Wi2}, Wise uses the
terms squared presentations and VH presentations for our \squarep s
and \gridp s.
\end{remark} 

We may depict a relator $r$ of a grid presentation by a Euclidean
square as follows: 
\[
\labellist
\small\hair 2pt
\pinlabel $c$ at  54 104
\pinlabel $b$ at  102 53
\pinlabel $a$ at  54 5
\pinlabel $d$  at 5 53
\endlabellist
\raisebox{8mm}{$a b^{-1} c d^{-1} \quad \longleftrightarrow  \quad $}\includegraphics[height=20mm]{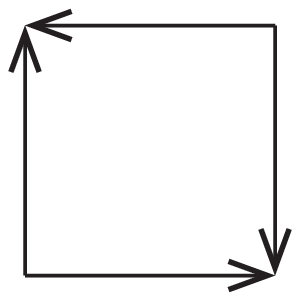}
\]

It is easy to see that in a grid presentation the following holds:
\begin{itemize}
\item
Relator squares have oriented edges, labelled from $X$. There are two
sinks and two sources in each relator square.
\item
We call a two letter subword of a relator a {\em pair}.  The 
\canc4 condition says that a pair uniquely determines a relator  up 
to cyclic permutation and inversion.
\item
\cant4 says that if $ab$ and $b^{-1}c$ are pairs then $ac$ is not.
\item
If $a$,  $b$ and $c$ are 
letters such that $ab$ and $b^{-1}c$ are both pairs (with $b\neq c$), 
then the word $ac$ is called a  {\em sister-set}. By the $T(4)$ condition, 
no pair is a sister-set.
\item
The edges of a relator square have an additional coloring: they are 
vertical or horizontal. Moreover, going around a relator square we
alternate between black and white.
\item We can invoke a convention that the black and white colorings correspond to horizontal and vertical line placement in our drawings of relator squares.

\item
A rotation or  reflection of a relator square corresponds to the cyclic
permutation or inversion of a relator. 
\end{itemize}

We can now state Weinbaum's theorem.

\begin{theorem}
\label{thm.Wn}\cite{Wn}
The augmented Dehn presentation of a 
prime, reduced, alternating knot diagram is a \gridp .
\end{theorem}

In \cite{LS} Lyndon and Schupp show that square and grid presentations have 
have solvable word and conjugacy problems.
Since the appearance of that work, polynomial time algorithms have  been given for the 
word (see \cite[Sec.7]{Jg1}))  and conjugacy problems (\cite{Jg1}) of these groups. 
We use these more efficient algorithms here. 

\subsection{The word problem for \squarep s}
\label{sub.wordp}

In this section we  recall the solution to the word problem of
\squarep s. 
To any group presentation $G= \langle X | R \rangle $ we can associate 
a {\em standard 2-complex} $K$ in the usual way: $K$ consists of one 
0-cell, one labelled 1-cell for each generator and one 2-cell for each 
relator, where the 2-cell $D_r$ representing the relator $r \in R$ is 
attached to the 1-skeleton, $K^{(1)}$, by a continuous map which identifies the boundary 
$\partial D_r$ with a loop representing $r$ in the 1-skeleton.
We impose a piece-wise Euclidean structure on the standard 2-complex and 
set all 1-cells to be of unit length.

A  word $w$ represents the identity in  $G$ if and only if there 
is a  simply connected planar 2-complex $\PD$, and a map
$\phi : (D,\partial D) \rightarrow (K, K^{(1)})$ such that the 0-cells are 
mapped to  0-cells, open $i$-cells are mapped to open $i$-cells, for 
$i=1,2$ and $\partial D$ is mapped to the loop representing $w$ in $K^{(1)}$.
Such a 2-complex, labelled in the natural way, is called a {\em Dehn diagram}. 

Throughout this text we  use two concepts of labels of edge-paths of 
the standard 2-complex, \pcomplex\ (introduced below)
or Dehn diagram. The label of 
an edge-path is the sequence of letters determined by the edge-path, 
where travelling along an edge labelled $a$ contributes the letter $a$.  
This is distinct from the word labelling an edge-path, which is the word 
in the group determined by the path, where travelling along an edge 
labelled $a$ against the orientation contributes the letter $a^{-1}$, 
and travelling with the orientation, the letter $a$. 

A word in a group presentation is said to be {\em geodesic} if it contains 
the least number of letters over all representatives of the same word, 
i.e., $w$ is geodesic if $|w|=\min \{ |w'| \; | \; w=_G w' \}$. A geodesic 
word represents the identity if and only if it is the empty word. A word 
in a group presentation is geodesic if and only if it labels a geodesic 
edge-path in the standard two complex  of the presentation.

A key result of small cancellation theory is the following
{\em Geodesic Characterisation Theorem}; see \cite[Sec.3]{Jg1} and also 
\cite[Lem.3.2]{Ka}. 

\begin{theorem}
\label{thm:GCT}
A word in a \squarep\ 
is geodesic if and only if it is freely reduced and contains no subword 
$x_1 \dots x_n$
which is part of a {\em chain}:
%
\[
\labellist
\small\hair 2pt
\pinlabel $x_1$ at  54 54
\pinlabel $x_2$ at  108 104
\pinlabel $x_3$ at  180 104
\pinlabel $x_{n-1}$  at 323 104
\pinlabel $x_n$ at  378 54
\endlabellist
\includegraphics[height=20mm]{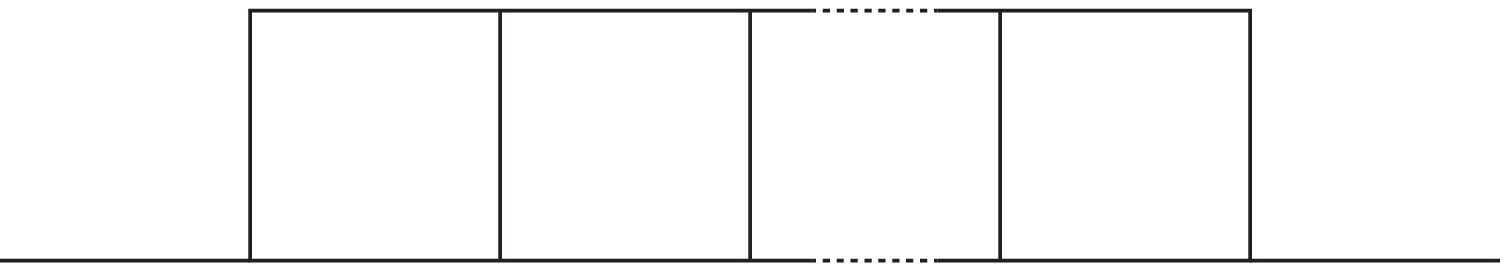}
\]

The  word $x_1 \dots x_n$ is called a {\em chain word}.
\end{theorem}
 
\begin{remark}
Observe that the Geodesic Characterisation Theorem immediately provides 
a quadratic time solution for the word problem in \squarep s: Given a 
word $w$, freely reduce it to obtain a word $w'$. If $w'$ is the empty 
word then $w=_G 1$, otherwise search $w'$ for a chain word. If $w'$  
does not contain a chain word then $w' \neq_G 1$. If $w'$ does contain 
a chain word, replace it with the shorter word which bounds the 
``other side'' of the chain to obtain a shorter word $w''$. Repeat the 
above process with the word $w''$ in place of $w$.    
\end{remark}

\subsection{The \pcomplex}
\label{sub.log-cabin}

A reduced, prime,  alternating, oriented knot diagram gives rise
to a \gridp\ with solvable word problem; see Theorems \ref{thm.Wn} and 
\ref{thm:GCT}. This \gridp\ contains a {\em peripheral $\BZ^2$-subgroup} 
generated by the meridian $m$ and the longitude $l$ of the knot. Of course, 
a peripheral subgroup does not exist for a general \gridp . 

Theorem \ref{thm.arcs} requires us to solve the {\em peripheral word problem}. 
Following Johnsgard (see \cite[Sec.7]{Jg1}), we consider 
the (rather overlooked) \emph{\pcomplex}, and we discuss how it solves
the peripheral word and conjugacy problem. 

Let $\PD$ be a reduced, prime,  alternating, oriented knot diagram with $n$ 
crossings, and let $\A_{\PD}$ be its augmented Dehn presentation. Each relator of 
$\A_{\PD}$ is a word of length four whose exponents alternate in sign. We may 
think of the relators as $1\times 1$ Euclidean squares with directed and 
labelled edges. For convenience, we impose some conventions upon our 
construction. We  discuss the effect of these conventions in 
Remark~\ref{rem:lg} below.

From the base point of $\PD$ and in the direction of the orientation, walk 
around the diagram and  label the $n$ crossings of $\PD$ with $c_1, c_2, 
\ldots , c_{2n}$ in the order we meet them and  in such a way that the 
label $c_1$ is assigned to the first under crossing we meet. 
For example, for the $5_2$ knot we have:
%
\begin{center}
\labellist
\small\hair 2pt
\pinlabel $c_1$ at 95 95
\pinlabel $c_2$ at 175 21
\pinlabel $c_3$ at 207 58
\pinlabel $c_4$ at 182 143
\pinlabel $c_5$ at 40 140
\pinlabel $c_6$ at 80 112
\pinlabel $c_7$ at 177 120
\pinlabel $c_8$ at 232 50
\pinlabel $c_9$ at 156 35
\pinlabel $c_{10}$ at 67 138
\pinlabel $0$ at 210 160
\pinlabel $1$ at 50 115
\pinlabel $2$ at 68 64
\pinlabel $3$ at 200 25
\pinlabel $4$ at 225 100
\pinlabel $5$ at 110 142
\pinlabel $6$ at 160 77
\endlabellist
\includegraphics[height=5cm]{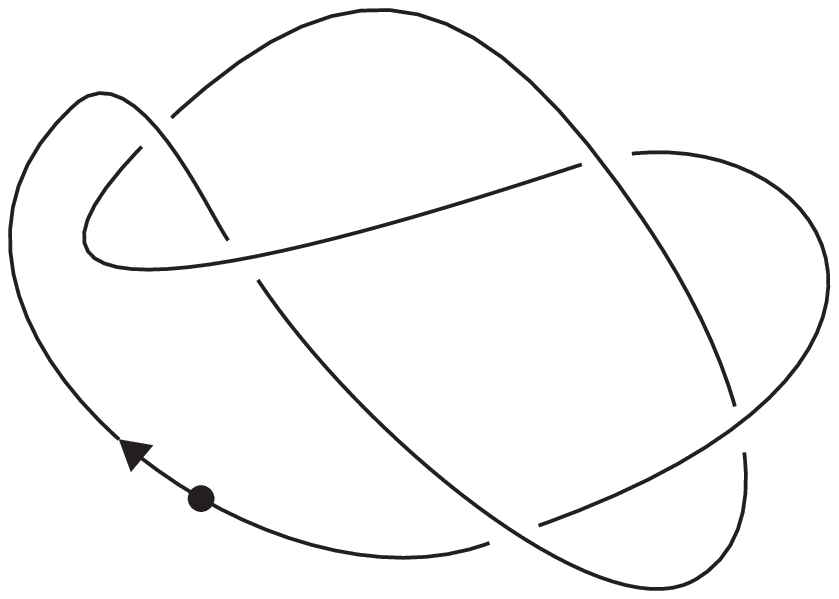}
\end{center}

We construct a $2n\times 1$ rectangle made out of $2n$ 
relator squares inductively as follows.
Position the relator square $C_1$ on the Euclidean plane in such a way 
that the label of the edge-path from $(0,0)$ to $(1,1)$ describes a loop 
which follows the knot through the undercrossing at $c_1$ (on the left is shown
the crossing $c_1$ and on the right is shown the relator square $C_1$):

\begin{center}
\labellist
\small\hair 2pt
\pinlabel $a$ at  35 71
\pinlabel $b$ at  72 71
\pinlabel $c$ at  72 35
\pinlabel $d$  at 35 35
\endlabellist
\raisebox{-3mm}{\includegraphics[height=25mm]{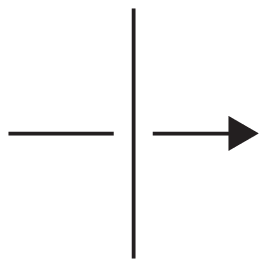}}
\quad
\raisebox{8mm}{$\leadsto$}
\qquad
\labellist
\small\hair 2pt
\pinlabel $c$ at  54 104
\pinlabel $b$ at  102 53
\pinlabel $a$ at  54 5
\pinlabel $d$  at 5 53
\endlabellist
\includegraphics[height=20mm]{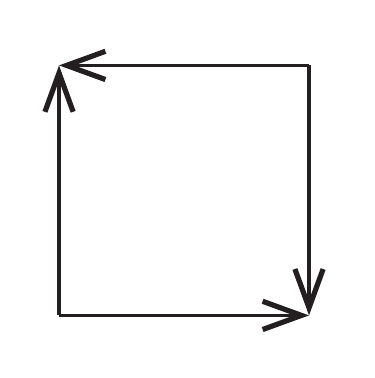}
\quad
\raisebox{8mm}{or}
\quad
\labellist
\small\hair 2pt
\pinlabel $b$ at  54 104
\pinlabel $c$ at  102 53
\pinlabel $d$ at  54 5
\pinlabel $a$  at 5 53
\endlabellist
\includegraphics[height=20mm]{lc1}
\end{center}

Suppose we have placed a relator square $C_k$ (which  arises from the 
crossing $c_k$). The relator squares $C_k$ and  $C_{k+1}$ have  exactly two 
edge-labels in common (since the diagram $\PD$ is prime and reduced).   
Identify the right edge of $C_k$ with the unique edge of $C_{k+1}$ which 
has the same label in a way that  preserves the orientation of the edges. 
This gives a $(k+1)\times 1$ rectangle.
Continue this process until  we have added the relator square $C_{2n}$. 
  
We call such a  $2n\times 1$ rectangle of relator squares a 
{\em fundamental block} of $\PD$. For example, the fundamental block of 
the $5_2$ knot above is
%
\begin{center}
\labellist
\small\hair 2pt
\pinlabel $6$ at  50 104
\pinlabel $0$ at  121 104
\pinlabel $6$ at  194 104
\pinlabel $0$  at 266 104
\pinlabel $1$ at   341 104
\pinlabel $6$ at   412 104
\pinlabel $0$ at  487 104
\pinlabel $6$  at 558 104
\pinlabel $0$ at  634 104
\pinlabel $1$  at 698 104
\pinlabel $1$ at  50 4
\pinlabel $6$ at  121 4
\pinlabel $0$ at  194 4
\pinlabel $6$  at 266 4
\pinlabel $0$ at   341 4
\pinlabel $1$ at   412 4
\pinlabel $6$ at  487 4
\pinlabel $0$  at 558 4
\pinlabel $6$ at  634 4
\pinlabel $0$  at 698 4
\pinlabel $5$ at  7 54
\pinlabel $2$ at  81 54
\pinlabel $3$ at  153 54
\pinlabel $4$  at 224 54
\pinlabel $5$ at  296  54
\pinlabel $2$ at   370 54
\pinlabel $5$ at  442 54
\pinlabel $4$  at 514 54
\pinlabel $3$ at  586 54
\pinlabel $2$  at 657 54
\pinlabel $5$  at 728 54
\endlabellist
\includegraphics[height=20mm]{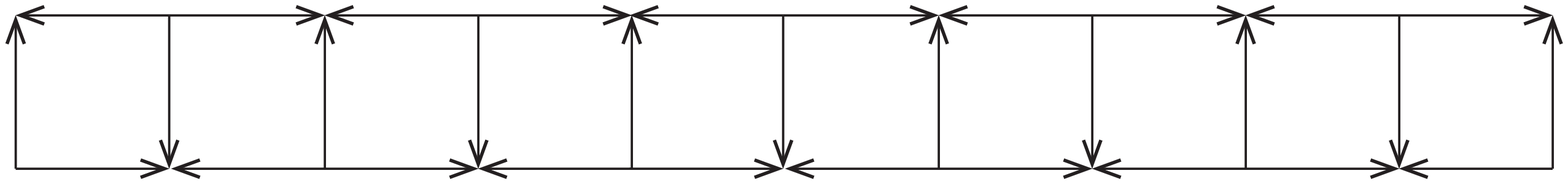}
\end{center}

Observe that the fundamental block has oriented edges and its vertices are 
either sinks or sources. We will often simplify figures by drawing sinks as 
thickened black vertices. This determines the orientations of the edges.

Notice in the example above that the word labelling  the top edge of the 
fundamental block is a cyclic permutation of the  word labelling  the 
bottom edge of the fundamental block, and that the labels and orientations on 
the left and right edges coincide. This observation holds in general and it 
allows us to piece together the fundamental blocks in a way that tiles 
the plane.

\begin{lemma}
\label{lem:lc1}
In the fundamental block of a  reduced, prime,  alternating, oriented knot 
diagram $\PD$,
\begin{enumerate}
\item the label and orientation of the rightmost and leftmost vertical 
edges of the fundamental block coincide;
\item the label and orientation on the top of the relator square $C_i$ 
is the same as the label and orientation on the bottom of the relator 
square $C_{i+1}$, where  the indices are taken modulo $2n$.
\end{enumerate}
\end{lemma}

We defer the proof of this lemma until the end of Section 
\ref{sub.propertieslog}. 

Using Lemma~\ref{lem:lg1} we can piece together 
together the fundamental blocks according to the following pattern,
\begin{center}
\includegraphics[height=0.10\textheight]{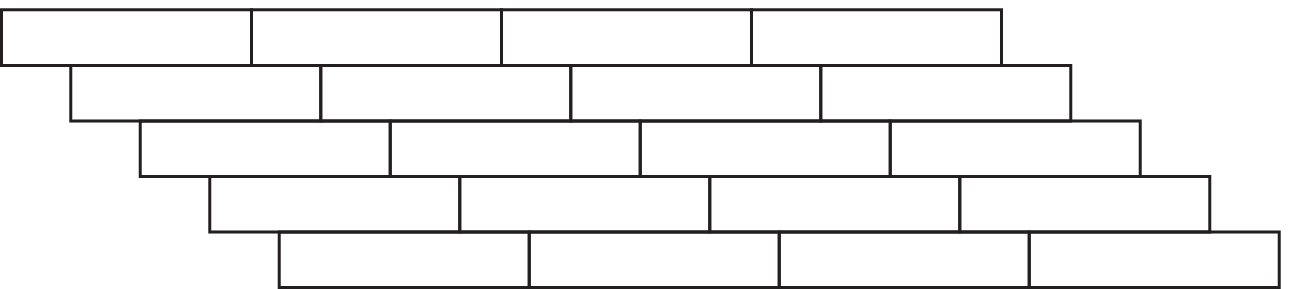}
\end{center}
and tile the whole plane by relator squares. 

\begin{definition}
\label{def.log-cabin}
We call the resulting $2$-dimensional CW complex the 
{\em \pcomplex}. 
\end{definition}
For example, a portion of the \pcomplex\ for the $5_2$ knot is given by:
\begin{center}
\includegraphics[height=0.35\textheight]{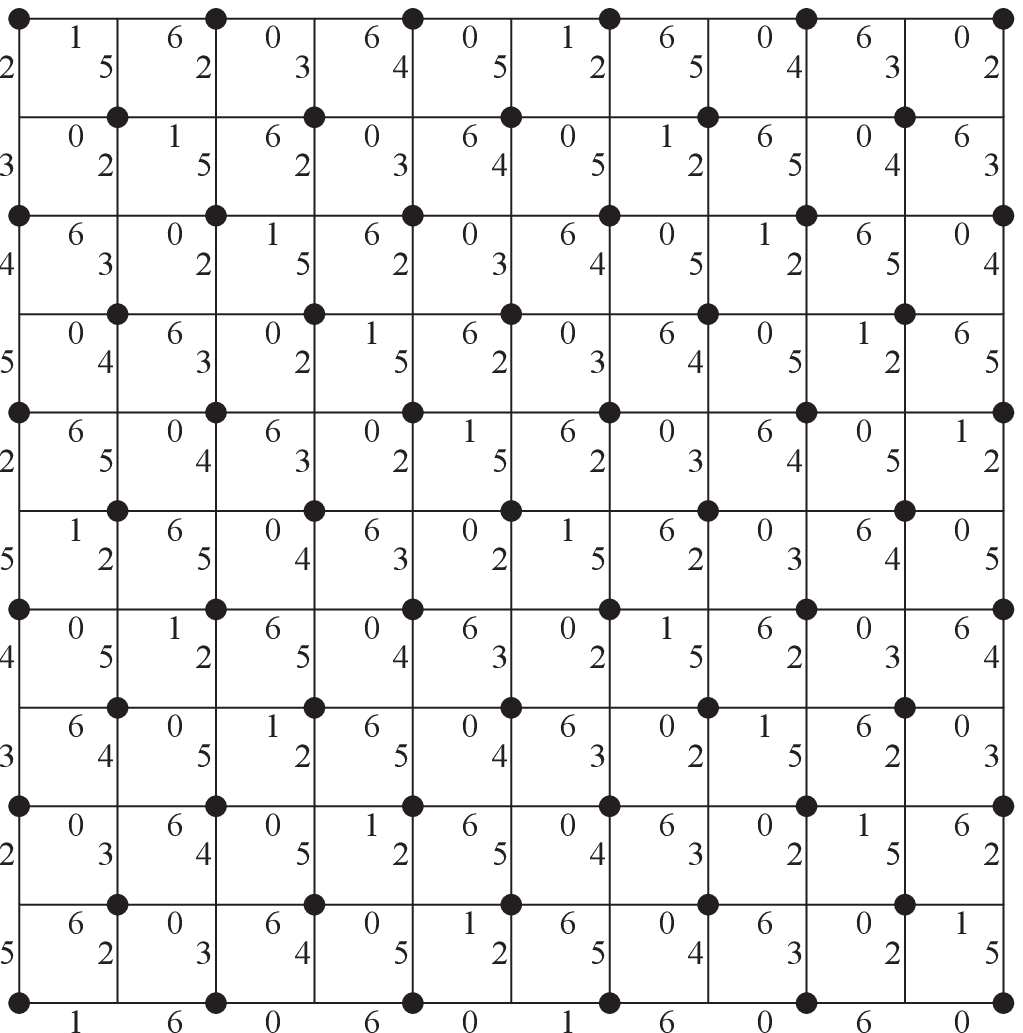}
\end{center}

\begin{remark}
\label{rem:lg}
Several choices and conventions were made in the construction of the 
\pcomplex. Namely the choice of base point on $\PD$,  the label $c_1$ was 
assigned to the first under crossing we met, and the positioning of the 
first relator square $C_1$. 
 It is clear from the construction of the complex that a different choice 
of base point (as well as orientation) and  a different placement of $C_1$ 
on the Euclidean plane would result in a \pcomplex\ which is 
isometric to the one constructed here.   We  discuss this in more detail 
in Section~\ref{sub.logcabin}. 

All of the arguments presented here can be made  with any construction of 
the \pcomplex, however the directions specified in the statements of 
results and proofs in this paper may change.  
\end{remark}

\subsection{Some properties of the \pcomplex}
\label{sub.propertieslog}

By construction, the \pcomplex\ embeds in the standard 2-complex  
of the augmented Dehn presentation. In fact it embeds geodesically:

\begin{lemma}
\label{lem.geodesic}
The \pcomplex\ of $\PD$  embeds geodesically in the standard 
2-complex of the augmented Dehn presentation $\A_{\PD}$.
In particular, the word labelling any geodesic edge-path in the \pcomplex\ 
is a geodesic word in the augmented Dehn presentation.
\end{lemma}

\begin{proof}
The proof uses the Geodesic Characterisation Theorem (Theorem~\ref{thm:GCT}).
Since any two paths in the \pcomplex\ with common beginning and 
ending
represent the same word in $\A_{\PD}$, it suffices to show that a path $p$ 
that goes
horizontally $q$ steps and then vertically $r$ steps  in the  
\pcomplex\ is geodesic in the
standard 2-complex.

Consider an edge-path $p$ in the \pcomplex\ that goes horizontally $q$
steps and then vertically $r$ steps. Such a path has at most one pair 
subword, since a sister-set is never a pair by the $T(4)$ condition. 
Thus we see that the label of the edge-path cannot contain a chain word, 
as this requires two pairs. (Recall the definitions of pairs and sister sets from Section~\ref{sub.squaregrid}.)

It remains to show that the label of the edge-path is freely reduced.  
To see why this is we begin by observing that since $\PD$ is reduced, four 
distinct regions of $\PD$ meet at every crossing and therefore every 
relator square has four distinct labels. Now suppose that $ab$ is a 
subword  of the word labelling the edge-path $p$. If the subword belongs 
to the horizontal path, then it labels the bottom of two relator squares  
$D_i$ and $D_{i+1}$ in the \pcomplex. By Lemma~\ref{lem:lc1}, the 
bottom label of  $D_{i+1}$ is also a label of the top of the relator square 
$D_i$.
This means that $b$ cannot label the bottom of $D_i$ and therefore  
$a\neq b$ and $ab$ is freely reduced. 

If the letter $a$ comes from a horizontal edge and $b$ from a vertical 
edge of $p$, then the subword $ab$ labels two sides of a relator square 
and is therefore freely reduced.

Finally, If the subword belongs to the vertical path, then it labels the 
right hand side  of two relator squares  $D_i$ and $D_{i+1}$ in the 
\pcomplex. By the periodicity of the \pcomplex, the 
right hand  label of  $D_{i}$ is also the label of the left hand side  of 
the relator square $D_{i+1}$.  This means that $b$ cannot label the right 
of $D_{i+1}$ and therefore  $a\neq b$ and $ab$ is freely reduced. 
\end{proof}

\begin{proof}[Proof of Lemma~\ref{lem:lc1}]
Since the exponents of the relators of the augmented Dehn 
presentation  alternate in sign, all of the orientations of the edges 
of the fundamental block are of the form required by the lemma. 

It remains to show that the edge labels are of the required form. 
First we show that 
the label on the top of the relator square $C_i$ is the same as the 
label  on the bottom of the relator square $C_{i+1}$ for $i=1, \ldots , n$.

Consider the relator square $C_1$ positioned as
%
\begin{center}
\labellist
\small\hair 2pt
\pinlabel $b$ at  54 104
\pinlabel $a$ at  102 53
\pinlabel $c$ at  54 5
\pinlabel $d$  at 5 53
\pinlabel $C_1$  at 53 53
\endlabellist
\includegraphics[height=20mm]{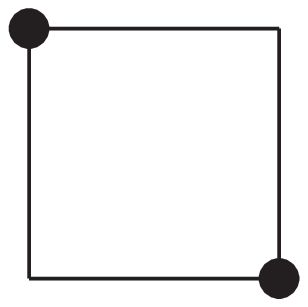}
\end{center}

By convention the labels $a$ and $b$ also appear in $C_2$ (since the 
regions $a$ and $b$ of $\PD$ are incident with the crossings $c_1$ and $c_2$). 
Therefore $C_2$ has one of the following forms 
\begin{center}
\labellist
\small\hair 2pt
\pinlabel $b$ at  54 104
\pinlabel $a$  at 5 53
\pinlabel $C_2$  at 53 53
\endlabellist
\includegraphics[height=20mm]{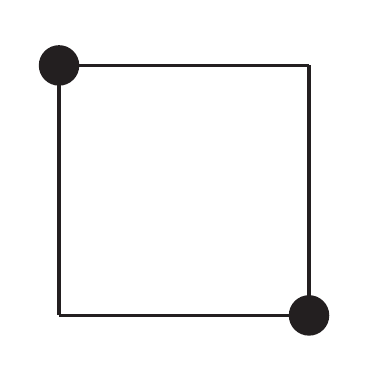}
\quad \raisebox{8mm}{or}\qquad
\labellist
\small\hair 2pt
\pinlabel $b$ at  54 5
\pinlabel $a$  at 5 53
\pinlabel $C_2$  at 53 53
\endlabellist
\includegraphics[height=20mm]{c}
\end{center}
These two relators have edge-paths $b^{-1}a$ and $ab^{-1}$ respectively.
By the small cancellation conditions, a  pair  uniquely determines a 
relator, so the word $ab^{-1}$ cannot appear in $C_2$  (as it appears 
in $C_1$ as $(ab^{-1})^{-1}$).  Therefore $b$ must be the label on the 
bottom of $C_2$.

We proceed inductively. Suppose that we have shown that the label on 
the top of $C_{k-1}$ coincides with the label on the bottom of $C_k$.
Since the crossing $c_k$ shares two incident regions with $c_{k-1}$ and 
two with $c_{k+1}$, the relator square $C_k$ share two labels with  
$C_{k-1}$ and two with $C_{k+1}$. By hypothesis, $C_k$ shares the labels on 
the bottom and left-hand  edges with $C_{k-1}$, so it  shares the labels 
on the top and right-hand edges with $C_{k+1}$.

Suppose  the word on the edge-path which follows the right-hand and then 
top edge of $C_k$ is $rs^{-1}$ or $r^{- 1}s$. We will deal with each case 
separately.

If the path is  $rs^{-1}$. 
Then $C_{k+1}$ also has edges labelled $r$ and $s$  and must contain the 
word $s r^{- 1}$ or $r^{-1}s$.  Since a pair determines a relator and 
$C_k \neq C_{k+1}$, we have that   $C_{k+1}$ must contain the word $r^{-1}s$ 
(since $s r^{- 1} = (rs^{-1})^{-1}$). The only way this can happen is if the 
letter $s$ is on the bottom of $C_{k+1}$.

Similarly,  if the path is  $r^{-1}s$. 
Then $C_{k+1}$ also has edges labelled $r$ and $s$  and must contain the 
word $s^{-1} r$ or $rs^{-1}$.  Since a pair determines a relator and 
$C_k \neq C_{k+1}$,  we have that $C_{k+1}$ must contain the word $rs^{-1}$. The only way 
this can happen is if the letter $s^{-1}$ is on the bottom of $C_{k+1}$.

We have shown that the label on the top of the relator square $C_i$ is 
the same as the label  on the bottom of the relator square $C_{i+1}$ for 
$i=1, \ldots , n$. 

To complete the proof, consider the relator square 
$C_{2n}$ in the fundamental block. $C_{2n}$ is of the form
%
\begin{center}
\labellist
\small\hair 2pt
\pinlabel $s$ at  54 104
\pinlabel $r$ at  102 53
\pinlabel $q$ at  54 5
\pinlabel $p$  at 5 53
\pinlabel $C_{2n}$  at 53 53
\endlabellist
\includegraphics[height=20mm]{c}
\end{center}
%
where the labels $p$ and $q$ are shared with $C_{2n-1}$ and $r$ and $s$ 
are shared with $C_1$ (since $c_1$ and $c_{2n}$ share incident regions of 
$\PD$).  
But again, the small cancellation conditions say that a pair uniquely 
determines a relator and $C_{2n}\neq C_1$, therefore we must have  $s=c$ 
and $r=d$, where $c$ and $d$ are the labels of $C_1$ as shown above.  
This completes the proof of the lemma.
\end{proof}

\begin{lemma}
\label{lem:lc2}
In the fundamental block of a  reduced, prime,  alternating, oriented knot 
diagram $\PD$,
\begin{enumerate}
\item  
the label of an edge-path from the bottom-right to top-left 
corner of $C_{2n}$ describes a curve homotopic to a meridional loop, 
or its inverse, of the knot through the base point of $\PD$;
\item 
the label of an  edge-path from the bottom left to top right 
corner of $C_{2i-1}$, $i=1, \ldots , n$ describes a loop which follows 
the under-crossing  of the knot at $c_{2i-1}$; 
\end{enumerate}
\end{lemma}

\begin{proof}
The relator square $C_{2n}$ comes from a crossing of the form
%
\begin{center}
\labellist
\small\hair 2pt
\pinlabel $a$ at  35 71
\pinlabel $b$ at  72 71
\pinlabel $c$ at  72 35
\pinlabel $d$  at 35 35
\endlabellist
\raisebox{-3mm}{\includegraphics[height=25mm]{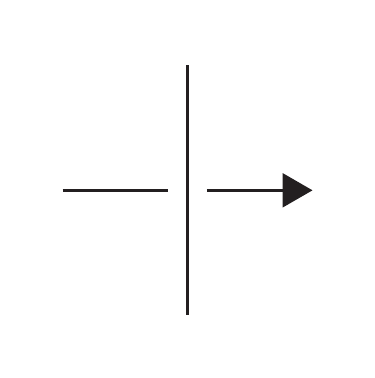}}
\end{center}
%
and, since $b$ and $c$ are also labels of $C_1$, it must appear in the 
fundamental block in one of the following forms
%
\begin{center}
\labellist
\small\hair 2pt
\pinlabel $b$ at  54 104
\pinlabel $c$ at  102 53
\pinlabel $d$ at  54 5
\pinlabel $a$  at 5 53
\endlabellist
\includegraphics[height=20mm]{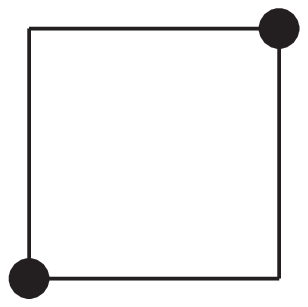}
\quad \raisebox{8mm}{or}\qquad
\labellist
\small\hair 2pt
\pinlabel $c$ at  54 104
\pinlabel $b$ at  102 53
\pinlabel $a$ at  54 5
\pinlabel $d$  at 5 53
\endlabellist
\includegraphics[height=20mm]{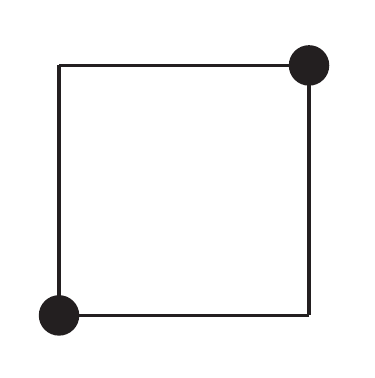}
\end{center}
In either case  we see that the edge-path bottom-right  to top-left corner 
describes a meridian or its inverse.  This proves the first statement of 
the lemma.

We now prove the second statement. The  relator square $C_{2i-1}$ appears 
in the fundamental block with orientation
%
\begin{center}
\includegraphics[height=20mm]{c}
\end{center}
and  at $c_{2i-1}$ we travel along an undercrossing of the form
%
\begin{center}
\labellist
\small\hair 2pt
\pinlabel $a$ at  35 71
\pinlabel $b$ at  72 71
\pinlabel $c$ at  72 35
\pinlabel $d$  at 35 35
\endlabellist
\raisebox{-3mm}{\includegraphics[height=25mm]{a}}
\end{center}

which contributes the relator $ab^{-1}cd^{-1}$.
Therefore $C_{2i-1}$ is of one of the four forms
\begin{center}
\labellist
\small\hair 2pt
\pinlabel $c$ at  54 104
\pinlabel $b$ at  102 53
\pinlabel $a$ at  54 5
\pinlabel $d$  at 5 53
\endlabellist
\includegraphics[height=20mm]{c}
\quad \labellist
\small\hair 2pt
\pinlabel $a$ at  54 104
\pinlabel $d$ at  102 53
\pinlabel $c$ at  54 5
\pinlabel $b$  at 5 53
\endlabellist
\includegraphics[height=20mm]{c}
\quad
\labellist
\small\hair 2pt
\pinlabel $d$ at  54 104
\pinlabel $a$ at  102 53
\pinlabel $b$ at  54 5
\pinlabel $c$  at 5 53
\endlabellist
\includegraphics[height=20mm]{c}
\quad 
\labellist
\small\hair 2pt
\pinlabel $b$ at  54 104
\pinlabel $c$ at  102 53
\pinlabel $d$ at  54 5
\pinlabel $a$  at 5 53
\endlabellist
\includegraphics[height=20mm]{c}
\end{center}
(these are all possible ways that the relator can fit the orientation 
of $C_{2i-1}$). But, by the construction of the fundamental block, $a$ or $d$ 
must label the vertical left edge of $C_{2i-1}$. This eliminates two of the 
four possible labellings  of $C_{2i-1}$ above, and it is easily seen that in 
the remaining two possibilities, the label of an  edge-path from the 
bottom left to top right corner  describes a loop which follows the 
under-crossing  of the knot at $c_{2i-1}$, as required.
\end{proof}

Let $n$ denote the number of crossings of $\PD$.
Further, let $\lambda$ denote the double of the diagram   $\PD$ determined by  the blackboard framing, and based at 
a point $x_0$. Let 
$\mu$ be the meridian of $\PD$  based at $x_0$.  The orientations of $\mu$ 
and $\lambda$ are determined by the orientation of $\PD$.
A peripheral element of the knot group $\pi_1(S^3-K)$ is then a product 
$\lambda^a \mu^b$, $a,b \in \mathbb{Z}$.  
The curves $\lambda$ and $\mu$ in $\PD$ also determine canonical elements 
$\phi^{-1}(\iota(\lambda))$ and $\phi^{-1}(\iota(\mu))$ of the augmented knot 
group $G=\pi_1(S^3-(K \cup \O))$.   We abuse notation and also denote these 
elements by $\lambda$ and $\mu$ respectively.
We say that an element of the (augmented) knot group is {\em peripheral} 
if it represents the  element 
$\lambda^a \mu^b$ for some  $a,b \in \mathbb{Z}$. 

\begin{lemma}
\label{lem:lg1}
Let $w$ be a word which labels an edge-path from the point $(0,0)$ to the 
point $(an-b , an+b  )$ in the \pcomplex\ of $\PD$, where 
$a,b\in \mathbb{Z}$ and $n$ is the number of crossings of $\PD$. Then 
$w$ is peripheral and represents the element $\lambda^a\mu^b$. Conversely, 
every peripheral element $\lambda^a\mu^b$ has a representative as the label 
of an edge-path from the point $(0,0)$ to the point $(an-b , an+b  )$ in 
the \pcomplex. 
\end{lemma}

\begin{proof}
We show that there exists one word $l^am^b$ labelling the edge-path from 
$(0,0)$ to $(an-b , an+b  )$ in the \pcomplex\ which represents 
$\lambda^a \mu^b$, for each choice of $a$ and $b$. Since the \pcomplex\ 
complex embeds in the standard  2-complex of the augmented Dehn 
presentation, it follows that any  word which labels a edge-path from the 
point $(0,0)$ to the point $(an-b , an+b  )$ represents the peripheral 
element $l^am^b$.

Label the crossings of $\PD$ by  $c_1, \ldots , c_{2n}$ according to the 
conventions in Subsection~\ref{sub.log-cabin}.   We can find a 
representative $l$ of $\lambda$ in the augmented Dehn presentation 
$\A_{\PD}$ as follows: 
take a framed double $\lambda$ of $\PD$.
Begin by taking $l$ to be the empty word. Walk once around $\lambda$ 
and concatenate a subword  $X_aX_b^{-1}$  to the right of  $l$ whenever we 
pass under an arc of $\PD$ from a region labelled $a$ to a region labelled 
$b$. 
The word $l$ obtained clearly represents $\lambda$.

Since the knot is alternating and, by our convention on the labelling of 
the crossing, the double $\lambda$ of $\PD$ passes under an arc of $\PD$ 
at the crossings $c_{2i-1}$, for $i=1, \ldots, n$. By Lemma~\ref{lem:lc2}, the 
two letter subword contributed to $l$ at the crossing $c_{2i-1}$ is exactly 
the label of an edge-path from the bottom left to top right corner of a 
relator square $C_{2i-1}$ in the \pcomplex. 
Therefore $l$ can be described as an edge-path from the bottom left to the 
top right of the following complex:
%
\begin{center}
\labellist
\small\hair 2pt
\pinlabel $C_1$ at  54 54
\pinlabel $C_3$ at  125 125
\pinlabel $C_{2n-1}$ at  270 270
\endlabellist
\includegraphics[height=50mm]{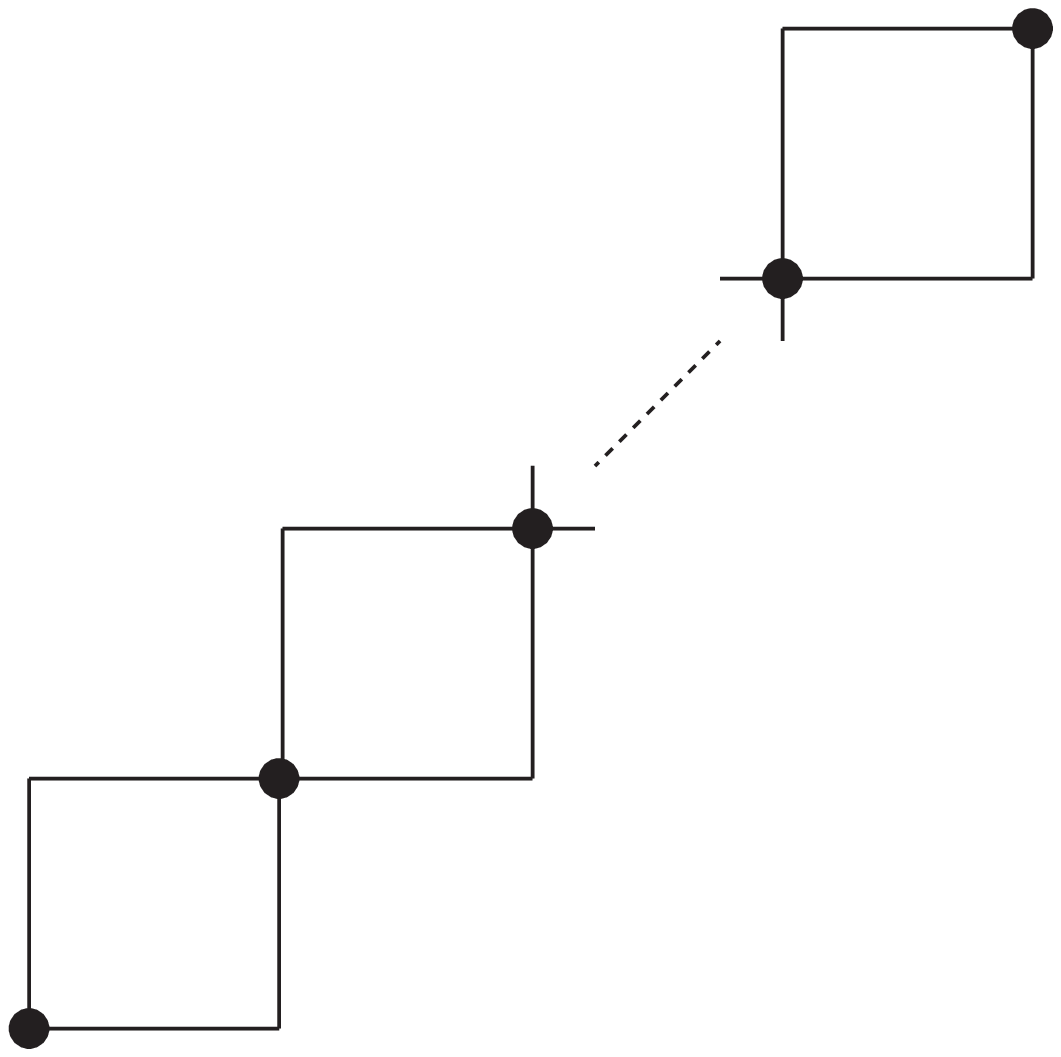}
\end{center}
Clearly, such a complex is embedded  in the \pcomplex\ and the edge-path is 
a path from $(0,0)$ to  $(n,n)$. By the periodicity of the 
\pcomplex, it follows that the word $l^a$ is an edge-path in the 
\pcomplex\ from $(0,0)$ to $(an,an)$ for all $a \in \mathbb{Z}$.

We have shown that powers of the longitude are contained in the 
\pcomplex. We now show that powers of meridians and peripheral 
elements are also contained in the \pcomplex.

By Lemma~\ref{lem:lc2} the label $m$ of an edge-path from $(0,0)$ to 
$( -1, 1)$ represents the  meridian  $\mu$ or its inverse $\mu^{-1}$. 
By the periodicity of the \pcomplex\ the label $m^b$ of an edge 
path from $(0,0)$ to $( -b, b)$, $b\in \mathbb{Z}$, represents a power 
of the meridian and, again by  the periodicity of the \pcomplex\ 
$(an,an)$ to $( an-b, an+b)$, $b\in \mathbb{Z}$ represents a power of the 
meridian for each $a,b\in \mathbb{Z}$. 

Therefore the label of an edge-path from the point $(0,0)$ to the point 
$(an-b , an+b  )$ in the \pcomplex\  is $l^am^b$ and  is peripheral. 
\end{proof}

\subsection{The peripheral word problem}
\label{sub.peripheralword}

The aim of this section is to solve the peripheral word problem using
the \pcomplex.

\begin{theorem}
\label{thm:lg2}
Let $w$ be a geodesic word in the augmented Dehn presentation of the $n$ 
crossing diagram $\PD$. Then $w$ is peripheral and represents the element 
$\lambda^a\mu^b$ if and only if it labels a geodesic edge-path from 
$(0,0)$ to $(an-b , an+b  )$ in the \pcomplex\ of $\PD$ for some  
$a,b\in \mathbb{Z}$.
\end{theorem} 

To prove the theorem we need the following result from \cite{Jg2} and 
\cite{Kr}.
\begin{theorem}
\label{thm:gcomp}
Let $w$ be a  geodesic word in a \squarep\ of 
a group $G$ all of whose relators are of length four. Then $w$  uniquely 
determines a tiling of relator squares  bounded by (but not necessarily 
filling) a rectangle in the Euclidean  plane such that:
\begin{enumerate}
\item the tiling embeds in the standard 2-complex of the group, i.e. it 
is a Dehn diagram;
\item the word labels a geodesic edge-path from one corner of the 
rectangle to the opposite corner; and 
\item if $w'$ is a geodesic word then $w'=_G w$ if and only if $w'$  
labels a geodesic edge-path from one corner of the rectangle to the 
opposite corner  path homotopic to $w$.  
\end{enumerate}  
\end{theorem}
The tiling produced by the theorem for a geodesic word $w$ is called 
the {\em geodesic completion} of $w$.

\begin{proof}[Proof of Theorem~\ref{thm:lg2}.]
Let $R_{ab}$ be the rectangle in the \pcomplex\ determined by the 
points $(0,0)$ and $(an-b , an+b  )$ for some integers $a$ and $b$, and 
let $w_{ab}$ be the label of any geodesic edge-path between these two 
points (for example the edge-path from $(0,0)$ to  $(an-b , 0  )$ to 
$(an-b , an+b  )$ will do). 
Then since the words labelling  geodesic edge-paths in the \pcomplex\ 
 are geodesic words in the augmented Dehn presentation 
(by Lemma~\ref{lem.geodesic}), $w_{ab}$ is a geodesic word.

Therefore, $w_{ab}$ is a geodesic word in a \gridp\
which labels  a geodesic edge-path between two opposite 
corners of  the rectangle $R_{ab}$. By Theorem~\ref{thm:gcomp},  a 
geodesic word in the augmented Dehn presentation represents the word 
$w_{ab}$ if and only if it is the label of a geodesic edge-path between   
$(0,0)$ and $(an-b , an+b  )$. 
 So all geodesic representatives of $w_{ab}$ are words labelling geodesic 
edge-paths from   $(0,0)$ and $(an-b , an+b  )$ in the \pcomplex.
 Finally,  by Lemma~\ref{lem:lg1}, every peripheral element is presented 
by a word   $w_{ab}$ for some $a, b \in \mathbb{Z}$ and the result follows.
\end{proof}

\subsection{Proof of Theorem \ref{thm.arcs}}
\label{sub.thm2}

Let $\PD$ be a prime, reduced alternating projection of a knot $K$ in
$S^3$. Theorem \ref{thm.arcs} follows from the following lemma.

\begin{lemma}
\label{lem:perarcs}
\rm{(a)}
The Wirtinger arcs, Wirtinger loops, and Dehn arcs and conjugates of
the short arcs of $\PD$ have explicit 
geodesic representatives in the \pcomplex.
\newline
\rm{(b)} The above geodesic representatives of Wirtinger arcs, Wirtinger 
loops and Dehn arcs are non-peripheral, and the above  geodesic 
representatives of the short arcs are not conjugate to a peripheral element.
\end{lemma}

\begin{proof}
We apply the notation and discussion from the last three paragraphs of 
Subsection~\ref{sub.Dehnp}.

Let $w$ be a word representing a Wirtinger arc, Wirtinger loop, short 
arc or Dehn arc of $\PD$. Recall the map $\phi$ from Equation \eqref{eq.phi}
and the method for reading the a representative  word 
$\phi^{-1}(\iota (w))$ in the augmented Dehn presentation described 
in Remark \ref{rem.phi}. We  use the \pcomplex\ to show that its 
image $\phi^{-1}(\iota (w))$ is non-peripheral. 

We deal with each type of loop separately. Throughout we let 
$l=l_1 l_2 \cdots l_{2n}$ be a geodesic representative of $\lambda$ which 
was constructed in the proof of Lemma~\ref{lem:lg1}. It is given by an 
edge-path following the sequence of relator squares $C_1,C_2, \ldots , 
C_{2n-1}$. Each two letter subword $l_{2i-1}l_{2i}$ of $l$ labels an edge 
path on $C_{2i-1}$. Also let
$m=m_1m_2$ be a geodesic representative of $\mu$ in the augmented Dehn 
presentation. By Theorem~\ref{thm:lg2}, $l$ and $m$ are labels of edge 
paths from $(0,0)$ to $(n,n)$ and $(0,0)$ to $(-1,1)$ respectively, 
in the \pcomplex.

\noindent{\bf Wirtinger arcs:}  
Wirtinger arcs are loops which follow the double $\lambda$ of $\PD$ 
returning to the base-point after  passing through  fewer than $2n$ 
crossings. Therefore a Wirtinger loop is represented by a subword 
$l_1 l_2 \cdots l_{2p}$, wher $2p<2n$, of $l$. Moreover,  $l_1 l_2 \cdots l_{2p}$ 
is represented by the label of any   edge-path from   $(0,0)$ to 
$(p,p)$ for $p<n$. By Theorem~\ref{thm:lg2}, it follows that 
$l_1 l_2 \cdots l_{2p}$ is non-peripheral as it does not label  a path 
between $(0,0)$ and $(an-b , an+b  )$.

\noindent{\bf Wirtinger loops:}  
We may move a Wirtinger loop close to some crossing $c_i$. By the way 
that the relators of the augmented Dehn presentation are read from $\PD$ 
(see Subsection~\ref{sub.Dehnp}), we see that the Wirtinger loop can be 
described by a geodesic edge-path between two opposite corners of $D_i$ 
(which two opposite corners depends on the given Wirtinger loop). 
Let the label of this edge-path be $w$. 

The word $w$ is geodesic of length two. Therefore, if $w$ is peripheral 
it must represent $\mu^{\pm1}$. However, by Lemma~\ref{lem:lc2}, the only 
geodesic words which represent $\mu^{\pm1}$ arise as a path from $(0,0)$ 
to $(\mp1,\pm1)$ in the peripheral complex,  so $w$ cannot be the label  
of such an edge-path since  by the definition of Wirtinger loops, $w$ is 
not a representative of the meridian, which is described by a path from 
$(0,0)$ and $(-1,1)$.

\noindent{\bf Dehn arcs:}  
Suppose that  a given Dehn arc intersects the bounded region $a$ of $\PD$. 
Then it is represented by $X_a X_0^{-1} $ in the augmented Dehn presentation.  
The word $X_a X_0^{-1}$  is geodesic since it is freely reduced ($a\neq 0$) 
and clearly does not contain a chain subword (see Theorem~\ref{thm:GCT}).

The meridian $\mu$ has exactly two geodesic representatives which label 
the edge-path $(0,0)$ to $(-1,1)$ of $C_{2n}$. It is easily seen from the 
definition of Dehn arcs that neither of these words can be  $X_a X_0^{-1}$. 

\noindent{\bf Short arcs:}  
Short arcs are found by walking around  the double $\lambda$ of $\PD$, and 
at some point, jumping to an adjacent arc of $\lambda$ and walking back 
to the base point in one of two ways. Short arcs are then represented 
by words of the form
\begin{equation}
\label{eq:3} l_1 l_2 \cdots l_{k}  l_p \cdots l_{2n}  
\end{equation}
and 
\begin{equation}
\label{eq:4} 
l_1 l_2 \cdots l_{k}  l_p \cdots  l_{1} . 
\end{equation}

These representatives of short arcs  do not necessarily embed as 
edge-paths in the \pcomplex. However, since $l^2$ embeds in 
the complex as a path from $(-n,-n)$ to $(n,n)$, the conjugates 

\begin{equation}
\label{eq:1}  
l_p \cdots l_{2n}   l_1 l_2 \cdots l_{k}   
\end{equation}
and 

\begin{equation}
\label{eq:2}  
l_p \cdots l_{1}   l_1 l_2 \cdots l_{k}   
\end{equation}
both embed in the \pcomplex\ as paths from 
$(-p,-p)$ to $(k,k)$, for (\ref{eq:1}), and $(q,q)$ to $(r,r)$ or  
$(r,r)$ to $(q,q)$, for (\ref{eq:2}), where $q=\min\{k,p\}$ and 
$r=max\{k,p\}$. These  paths travel along the  South-West to North-East axis.
 
Since these words in (\ref{eq:3}) and (\ref{eq:4}) are of length 
less that $2n$, if they are peripheral, then they must be equal to 
a power of the meridian $m^b$.  This will happen if and only if 
(\ref{eq:3}) and (\ref{eq:4})  represent 
\begin{equation}\label{eq:5}  (l_1 \cdots l_{k})^{-1}  m^b  
(l_1 \cdots l_{k})  , \end{equation}
for some integer $b$.
This  conjugate of $m^b$ embeds into the \pcomplex\ as a path 
from $( k,k )$ to $(0,0)$ to $(-b,b)$ to $(k-b,k+b)$
 which has a geodesic representative  as a path from $( k,k )$ to 
$(k-b, k+b)$.
But this path travels along the South-East to North-West axis. 
Therefore, by Theorem~\ref{thm:gcomp}, the geodesic words in   
(\ref{eq:3}) and (\ref{eq:4}) cannot be equal to the  words of the 
form in (\ref{eq:5}) and the short arcs are non-peripheral. 
\end{proof}

\begin{remark}
The argument above showing that short arcs are non-peripheral in fact proves a stronger result. It shows that any loop in the knot complement that follows the double of $\Delta$ from the basepoint, then at some point jumps (above the projection plane) to any other point on the double, then follows it back to the basepoint in either direction is non-peripheral.  
\end{remark}

\begin{remark}
\label{rem.conjugacy}
Notice that in the proof for the non-peripherality of short arcs we 
actually solved the conjugacy problem.  We could also have shown that 
these elements were non-peripheral by using Johnsgard's solution to 
the conjugacy problem  \cite{Jg1}: using Johnsgard's algorithm, the 
fact that the \pcomplex\ contains the geodesic completion of 
$l^m$, and the periodicity of the \pcomplex, it is straight-forward to show that a geodesic word in the augmented Dehn 
presentation is conjugate to a peripheral element $l^am^b$ if and only 
if it embeds as a geodesic path from $(0,k)$ to $(an+k-b , an+b)$, 
for some integer $k$. It is easy to see that two words in    
(\ref{eq:1}) and (\ref{eq:2})  are not of this form. We can use a 
similar argument for Wirtinger loops. 

Also note that this characterisation of conjugates of peripheral 
elements as paths in the \pcomplex\ provides a method for solving 
the peripheral conjugacy problem.
\end{remark}

\subsection{The \pcomplex \ and the Gauss code of an alternating knot}
\label{sub.logcabin}

In our proof of Theorem \ref{thm.arcs}, the peripheral complex
plays a key role, and encodes the peripheral structure of a prime, reduced, 
alternating projection of a knot. In this section we discuss additional 
properties of the peripheral complex and its relation with the Gauss code.

In Subsection~\ref{sub.log-cabin}, we constructed the peripheral complex by 
placing relator squares of the augmented Dehn presentation on the plane in a 
way determined by the oriented knot diagram. As previously noted,  some  
conventions were used in this construction. There is a way 
 to construct the peripheral complex 
directly  from the relators of the augmented Dehn presentation without any reference to the knot diagram:
\begin{enumerate}
\item Choose any relator square from the augmented Dehn presentation of a 
prime, reduced, alternating knot diagram, and place it in the Euclidean plane.

\item Choose two diagonally opposite vertices of this relator square, call 
them $a$ and $b$. Form a ``diagonal line'' of relator squares  by placing 
copies of the relator square in such a way that each vertex $a$ is identified 
with a vertex $b$ and all of the relator squares are translations of the first.

\item Complete the tiling by adding relator squares from the augmented Dehn 
presentation in a way consistent with the words labelling edge-paths. (Proposition~\ref{p:unorcom} tells us that this can be done in a unique way.)
\end{enumerate}
The construction is indicated in Figure~\ref{fig:unor}.
Throughout this section we  call this the {\em unoriented construction} 
of the peripheral complex, and we  refer to the complex constructed in 
Subsection~\ref{sub.log-cabin} as the {\em oriented construction}. We will 
also refer to the resulting  complexes as the unoriented and oriented 
peripheral complexes respectively.

%
\begin{figure}[!htpb]
\begin{center}
\labellist
\small\hair 2pt
\pinlabel $1$ at  44 170
\pinlabel $1$  at 190 97
\pinlabel $1$ at  260 170
\pinlabel $1$ at  333 243
\pinlabel $1$  at 477 97
\pinlabel $1$ at  549 170
\pinlabel $1$ at  622 243
\pinlabel $2$ at  549 97
\pinlabel $2$ at  622 170
\pinlabel $2$ at  694 243
\pinlabel $1$  at 837 97
\pinlabel $1$ at  910 170
\pinlabel $1$ at  980 243
\pinlabel $3$ at  980 97
\pinlabel $3$ at  1052 170
\pinlabel $3$ at  1125 243
\pinlabel $2$ at  910 97
\pinlabel $2$ at  980 170
\pinlabel $2$ at  1052 243
\endlabellist
\includegraphics[height=0.18\textheight]{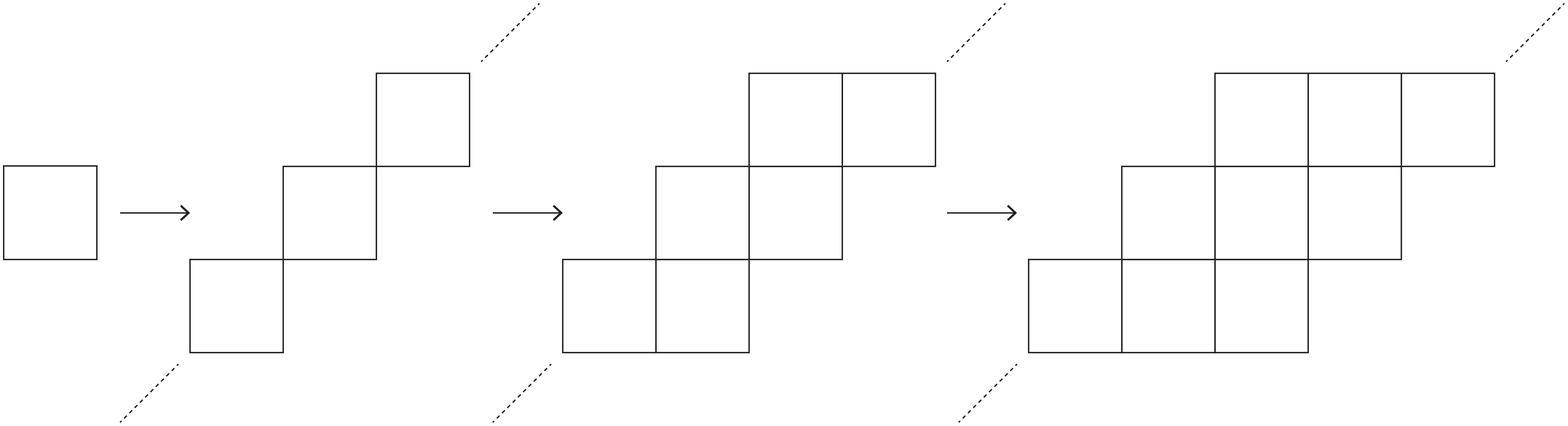}
\label{fig:unor}
\end{center}
\end{figure}



The following proposition tells us that the complex just described exists 
and is the peripheral complex.
\begin{proposition}
\label{p:unorcom}
The unoriented construction of the peripheral complex described above,  
produces  a unique plane tiling of relator squares. Moreover, the resulting 
complex is isometric to the oriented peripheral complex constructed in 
Subsection~\ref{sub.log-cabin}.  
\end{proposition}
\begin{proof}
First of all, we note that if the complex exists, then it must be unique, 
since the corners between pairs of relator squares in the ``diagonals'' 
used in the construction are labelled by  pairs and, in a grid presentation, 
a pair uniquely determines a relator.    

To show existence, let $T$ be the relator square in the plane from the first 
step of the  unoriented construction above. Suppose also that $T$ has  the 
vertices $a$ and $b$ specified. 
Since every relator square and the reflection of every relator square of the 
augmented Dehn presentation appears in the fundamental block of the oriented 
peripheral complex, there is an isometry taking $T$ to a relator square of 
the peripheral complex which sends the  vertices $a$ and $b$ to the top-left 
and bottom-right vertices of that relator square. By uniqueness, this extends 
to  an isometry of the complexes. 
\end{proof}

The following proposition tells us that an unoriented knot  can be 
recovered from its unoriented peripheral complex, and an oriented knot from 
its oriented peripheral complex.
\begin{proposition}
\label{prop.gauss}
Let $\PD$ be a prime, reduced, alternating, oriented  knot diagram. 
The Gauss code of $D$ can be recovered from the oriented peripheral complex; 
and the Gauss code of $\PD$ or its inverse $-\PD$ can be recovered from the  
unoriented peripheral complex. 
\end{proposition}

\begin{proof}
We first prove the result for the oriented complex. Choose any $2n\times 1$ 
horizontal block of the complex.  Every relator square appears exactly twice 
in this block. By the construction of the complex, this block is a cyclic 
permutation of a fundamental block, and therefore the order of the relator 
squares in the block is precisely the order we meet the crossings as we travel 
around the knot in the direction of the orientation from some base point. 
With this observation, it is straight-forward to recover the Gauss code:
label the relator squares $S_1, S_2, \ldots , S_{2n}$ by reading along the 
strip from left to right. 
Assign the number $-1$ to $S_1$ if it has orientation 
%
\begin{center}
\includegraphics[height=20mm]{c}
\end{center}
otherwise assign 
the number $+1$ to $S_1$. Suppose you have 
assigned the number $\pm j$ to the relator square $S_i$. If the relator 
square $S_{i+1}$ has not been encountered previously assign the number  
$\mp (j+1)$ to it, if the relator square has been encountered previously 
and has been assigned the number $\pm p$, then assign the number $\mp p$ 
to this square.  The resulting sequence is the Gauss code.

To recover a Gauss code from an unoriented peripheral complex, we can use the 
same method. However, since the $2n\times 1$ horizontal strip of the 
unoriented complex can be a reflection of a  $2n\times 1$ horizontal strip 
of the oriented complex, we are unable to determine if the Gauss code 
obtained is that of the knot diagram or its inverse.
\end{proof}


%


\bibliographystyle{hamsalpha}
\bibliography{biblio}
\end{document}